\documentclass[11pt,a4papar]{article}
\usepackage{amsthm}
\usepackage{amssymb}
\usepackage{amscd}
\usepackage{amsmath}
\usepackage[all]{xy}
\usepackage[top=30truemm,bottom=30truemm,left=25truemm,right=25truemm]{geometry}
\usepackage{comment}
\usepackage{algorithmic,algorithm}
\usepackage{here}
\usepackage{graphicx}
\usepackage{color} %
\usepackage{enumerate} %
\usepackage{listliketab}
\usepackage{multirow}
\usepackage{cases}


\makeatletter
  
  \@addtoreset{algorithm}{subsection}
\makeatother
\theoremstyle{definition}

\theoremstyle{definition}

\theoremstyle{plain}
\newtheorem{theo}{Theorem}%[section]

\theoremstyle{plain}
\theoremstyle{plain}
\theoremstyle{plain}
\theoremstyle{plain}

\theoremstyle{plain}
\newtheorem{thm}{Theorem}[subsection]%[subsection]
\theoremstyle{definition}
\newtheorem{ex}[thm]{Example}
\theoremstyle{definition}

\theoremstyle{definition}
\newtheorem{proc}[thm]{Procedure}
\theoremstyle{definition}

\theoremstyle{definition}

\theoremstyle{definition}
\newtheorem{rem}[thm]{Remark}
\theoremstyle{plain}
\newtheorem{prop}[thm]{Proposition}
\theoremstyle{plain}
\newtheorem{lem}[thm]{Lemma}
\theoremstyle{plain}

\theoremstyle{definition}

\theoremstyle{definition}

\theoremstyle{plain}

\theoremstyle{definition}

\theoremstyle{definition}

\numberwithin{equation}{section}

\def\qed{\hfill $\Box$}

\def\F{\mathbb{F}}

\def\deg{{\rm deg}}

\newcommand{\MSC}[1]{\textbf{{2010 Mathematical Subject Classification: }} #1}

\begin{document}

\title%[]
{\bf Genus-five hyperelliptic or trigonal curves with many rational points in characteristic three}
\author%[]
{Momonari Kudo\thanks{Graduate School of Information Science and Technology, The University of Tokyo.}
\ and Shushi Harashita\thanks{Graduate School of Environment and Information Sciences, Yokohama National University.}}

%Title
% \title{Genus-five hyperelliptic or trigonal curves with many rational points in characteristic three}
% \author{}
%% A note on certain superspecial and maximal curves of genus $5$}
% \author{Momonari Kudo and Shushi Harashita}
%% \thanks{Kobe City College of Technology.}
%% \thanks{Institute of Mathematics for Industry, Kyushu University. E-mail: \texttt{m-kudo@math.kyushu-u.ac.jp}}}
\providecommand{\keywords}[1]{\textbf{\textit{Key words---}} #1}
\maketitle

\begin{abstract}
The number $N_9(5)$, the maximal number of
$\F_9$-rational points on curves over $\F_9$ of genus $5$ is unknown, but
it is known that $32 \le N_9(5)\le 35$.
In this paper, we enumerate
hyperelliptic curves and trigonal curves over $\F_3$ which have
many $\F_9$-rational points (and $\F_3$-rational points),
especially the maximal number of $\F_9$-rational points of those curves is $30$.
Kudo-Harashita studied the nonhyperelliptic and nontrigonal case,
where they found a new example
of curves (over $\F_3$) of genus five which attains $32$
and proved that there is no example attaining more than $32$, among sextic plane curves with mild singularities. We conclude
from the main results in this paper that we need to search sextic models (i.e., nonhyperelliptic and nontrigonal) with bad singularities,
in order to find a genus-five curve over $\F_3$ with at least $33$ $\F_9$-rational points.
\end{abstract}

\footnote[0]{\keywords Algebraic curves, Rational points, Curves of genus five, Hyperelliptic curves, Trigonal curves\\
\MSC 14G05, 14G15, %14G50, 
14H50, 14Q05, 68W30}

%=====================
\section{Introduction}
%=====================

Throughout this paper, we use curve to mean a non-singular projective variety of dimension one.
Let $p$ be a rational prime, and let $\mathbb{F}_q$ denote the finite field of $q$ elements, where $q$ is a power of $p$.
For a curve $C$ of genus $g$ over $\mathbb{F}_q$, we denote by $\# C (K)$ the number of its $K$-rational points, where $K$ is a finite extension of $\mathbb{F}_q$.
Finding a curve $C$ with $\# C(\mathbb{F}_q)$ as large as possible (relative to $g$) is an interesting problem in its own right, and also such a curve is useful in applications such as coding theory (e.g., \cite{Hurt}). 
The most well-known (upper) bound on $\#C(\mathbb{F}_q)$ is the Hasse-Weil bound:
\begin{equation}\label{eq:HW}
 \# C(\mathbb{F}_q) \leq q + 1 + 2 g \sqrt{q}.
\end{equation}
% A curve attaining the Hasse-Weil upper bound is called an $\mathbb{F}_q$-maximal curve, but in general...
For further improvements of \eqref{eq:HW} when $g$ is larger with respect to $q$, see e.g., \cite{Ihara}, \cite{Serre} and \cite{Lauter}.
% Serre proved in the 1980s that
% \[
% \# C(\mathbb{F}_q) \leq q + 1 + 2 g \lfloor \sqrt{q} \rfloor.
% \]

Following Serre's paper~\cite{Serre}, we denote by $N_q(g)$ the maximal number of $\mathbb{F}_q$-rational points on a curve $C$ of genus $g$ over $\mathbb{F}_q$, namely,
\[
N_q(g) = \max \{ \# C (\mathbb{F}_q) : \mbox{$C$ is a curve of genus $g$ over $\mathbb{F}_q$}\},
\]
and then \eqref{eq:HW} implies $N_q(g) \leq q + 1 + 2 g \sqrt{q}$.
%For each $(q,g)$, the exact number $N_q(g)$ is one of the top concerns.
Here is a central problem in the study of algebraic curves over finite fields:
% Here an important problem is:
For a given $(q,g)$, determine (or bound) $N_q(g)$, and also find a curve of genus $g$ over $\mathbb{F}_q$ attaining $N_q(g)$, or its known best possible bound.
Indeed, the database of $N_q(g)$ has been developed at \cite{ManyPoints} since 2009, where the bound(s) or exact values of $N_q(g)$ have been updated by a number of researchers over the last 10 years.

%Serre introduced the number $N_q(g)$, the maximal number of
%$\F_q$-rational points on curves of genus $g$ over $\F_q$.
According to \cite{ManyPoints}, the ``smallest" case among the unknown cases at this point is $(q,g)=(9,5)$.
The maximal number of the $\F_9$-rational points of the known $\F_9$-rational points on curves of genus $5$ over $\F_9$ is $32$.
But the theoretical upper-bound of $N_9(5)$ is $35$ \cite{Lauter}, i.e, so far we know $32 \leq N_{9}(5) \leq 35$.

The purpose of this paper is to contribute to determine $N_{9}(5)$, with a complete search on the space of curves of genus $5$ over $\mathbb{F}_3$.
So far as is known, there are four concrete examples attaining $\#C(\mathbb{F}_9)=32$:
The first example was found by van der Geer - van der Vlugt \cite{GV}, and other two examples were known, one was found by Fischer 
(cf. \cite{ManyPoints}) and the other was found by Ramos-Ramos in \cite{Ramos-Ramos}.
In \cite{KH-genus5generic}, Kudo and Harashita found a new example
with enumeration of generic curves $C$ over $\mathbb{F}_3$ of genus $5$ such that $\# C (\mathbb{F}_9) = 32$,
where ``generic curve" means a non-hyperelliptic and non-trigonal
curve of genus $5$ with a sextic model in ${\mathbb P}^2$ with mild singularities, see \cite{KH-genus5generic} for the precise definition.
They also proved that there is no generic curve over $\mathbb{F}_3$ with $\# C (\mathbb{F}_9) > 32$.
%While Kudo and Harashita focused on the generic case, 
This paper investigates remaining cases: The hyperelliptic and trigonal cases.
Here are main theorems:
% The aim of this paper is to investigate the (non-)existence of trigonal or hyperelliptic curve of genus $5$ over $\mathbb{F}_3$ with many $\mathbb{F}_9$-rational points.
% For this, we give a reduction of defining equations of hyperelliptic and trigonal curves of genus $5$ in characteristic $3$.

\begin{theo}\label{thm:main2}
The maximal number of $\# H(\mathbb{F}_{9})$ of genus-five hyperelliptic curves $H$ over $\mathbb{F}_3$ is $20$.
Moreover, there are exactly $573$ $\mathbb{F}_{9}$-isomorphism classes of genus-five hyperelliptic curves $H$ over $\mathbb{F}_3$ with $20$ $\mathbb{F}_{9}$-rational points, and there are exactly $419$ $\mathbb{F}_9$-isogeny classes of Jacobian varieties among them.
(We omit to write their Weil polynomials here.)
In Example \ref{ex:hyp} (1) of Subsection \ref{subsec:MainHyp}, examples of genus-five hyperelliptic curves $H$ over $\mathbb{F}_3$ with $\# H (\mathbb{F}_{9}) = 20$ will be given.
\end{theo}

\begin{theo}\label{thm:main}
The maximal number of $\#C(\mathbb{F}_{9})$ of genus-five trigonal curves $C$ over $\mathbb{F}_3$ is $30$.
Moreover, there are exactly eight $\mathbb{F}_{9}$-isomorphism classes of genus-five trigonal curves $C$ over $\mathbb{F}_3$ with $30$ $\mathbb{F}_{9}$-rational points, whose Weil polynomials are
% \begin{eqnarray}
%  & & (t + 3)^4 (t^6 + 8 t^5 + 44 t^4 + 149 t^3 + 396 t^2 + 648 t + 729); \label{Weil:1} \\
%  & & (t^2 + 5 t + 9)(t^8 + 15 t^7 + 112 t^6 + 549 t^5 + 1927 t^4 + 4941 t^3 + 9072 t^2 + 10935 t + 6561); \label{Weil:2}
% \end{eqnarray}
\begin{equation}\label{Weil:1}
(t + 3)^4 (t^6 + 8 t^5 + 44 t^4 + 149 t^3 + 396 t^2 + 648 t + 729); 
\end{equation}
\begin{equation}\label{Weil:2}
(t^2 + 5 t + 9)(t^8 + 15 t^7 + 112 t^6 + 549 t^5 + 1927 t^4 + 4941 t^3 + 9072 t^2 + 10935 t + 6561);
\end{equation}
\begin{equation}
\begin{split}
& t^{10} + 20 t^9 + 196 t^8 + 1247 t^7 + 5714 t^6 + 19667 t^5 \\
 & + 51426 t^4 + 101007 t^3 + 142884 t^2 + 131220 t + 59049; 
 \end{split}\label{Weil:3}
\end{equation}
\begin{equation}\label{Weil:4}
(t^2 + 2 t + 9) (t^4 + 9 t^3 + 37 t^2 + 81 t + 81)^2;
\end{equation}
\begin{equation}\label{Weil:5}
\begin{split}
& t^{10} + 20 t^9 + 200 t^8 + 1299 t^7 + 6030 t^6 + 20843 t^5 \\
& + 54270 t^4 + 105219 t^3 + 145800 t^2 + 131220 t + 59049;
\end{split}
\end{equation}
\begin{equation}\label{Weil:6}
(t + 3)^2 (t^2 + 5 t + 9) (t^6 + 9 t^5 + 49 t^4 + 177 t^3 + 441 t^2 + 729 t + 729);
\end{equation}
\begin{equation}\label{Weil:7}
\begin{split}
& t^{10} + 20 t^9 + 194 t^8 + 1210 t^7 + 5433 t^6 + 18539 t^5 \\
& + 48897 t^4 + 98010 t^3 + 141426 t^2 + 131220 t + 59049;
\end{split}
\end{equation}
\begin{equation}\label{Weil:8}
(t^2 + 2 t + 9)(t^4 + 9 t^3 + 37 t^2 + 81 t + 81)^2.
\end{equation}
% \begin{eqnarray}
% & & (t^2 + 2 t + 9) (t^4 + 9 t^3 + 37 t^2 + 81 t + 81)^2; \label{Weil:4} \\
% & & t^{10} + 20 t^9 + 200 t^8 + 1299 t^7 + 6030 t^6 + 20843 t^5 + 54270 t^4 + 105219 t^3 + 145800 t^2 + 131220 t + 59049; \label{Weil:5}\\
% & & (t + 3)^2 (t^2 + 5 t + 9) (t^6 + 9 t^5 + 49 t^4 + 177 t^3 + 441 t^2 + 729 t + 729); \label{Weil:6}\\
% & & t^{10} + 20 t^9 + 194 t^8 + 1210 t^7 + 5433 t^6 + 18539 t^5 + 48897 t^4 + 98010 t^3 + 141426 t^2 + 131220 t + 59049; \label{Weil:7}\\
% & & (t^2 + 2 t + 9)(t^4 + 9 t^3 + 37 t^2 + 81 t + 81)^2. \label{Weil:8} 
% \end{eqnarray}
In the proof given in Subsection \ref{subsec:MainTri}, examples of genus-five trigonal curves $C$ over $\mathbb{F}_3$ with $\# C (\mathbb{F}_{9}) = 30$ will be given.
\end{theo}

We find from these theorems together with \cite[Theorem 1]{KH-genus5generic} that, in order to find a genus-five curve over $\mathbb{F}_3$ attaining $\#C(\mathbb{F}_9)>32$, there is no other way than searching sextic models (i.e., nonhyperelliptic and nontrigonal) given in \cite{KH-genus5generic} with bad singularities.

To prove Theorem \ref{thm:main2} (resp.\ Theorem \ref{thm:main}), we do a complete search over Magma~\cite{Magma}, \cite{MagmaHP} on the space of hyperelliptic (resp.\ trigonal) curves of genus $5$ over $\mathbb{F}_3$, with the following two ingredients:
\begin{enumerate}
\item An explicit formula for the number of rational points of hyperelliptic curves (resp.\ genus-five trigonal curves) is given in Subsection \ref{subsec:hyp} (resp.\ Subsection \ref{subsec:tri}).
\item To make the search terminate in real time, we also reduce the number of unknown coefficients in defining equations for our curves.
A reduction of defining equations for hyperelliptic curves (resp.\ genus-five trigonal curves) is presented in Subsection \ref{subsec:ReductionHyper} (resp.\ Subsection \ref{subsec:ReductionTrigonal}).
\end{enumerate}
The ingredients (1) and (2) for hyperelliptic curves are applicable for arbitrary genus $g \ge 1$ and arbitrary characteristic $\ne 2$.
For genus-five trigonal curves, (1) holds for arbitrary fields (characteristic not necessarily equal to $3$), whereas (2) works for every finite extension of $\mathbb{F}_3$.
Although the main theorems are argued for small finite fields ($\mathbb{F}_3$ and $\mathbb{F}_9$) of characteristic $3$, these ingredients might also derive fruitful applications for further study on algebraic curves over (other) fields.

The organization of this paper is as follows.
Section \ref{sec:pre} collects some results on counting rational points on a curve of genus $5$ with the classification of curves of genus $5$.
In Section \ref{sec:red} we give a reduction of hyperelliptic curves and a reduction of quintic models of trigonal curves (in characteristic $3$).
Section \ref{sec:main} is devoted to proving the main theorems (Theorems \ref{thm:main2} and \ref{thm:main}).

% In Section 1, we review general facts related to this work,
% especially a reduction of hyperelliptic curves
% and a reduction of quintic models of trigonal curves
% (in characteristic $3$).
% In Section 2, we state and prove
% the enumeration result in the hyperelliptic case.
% In Section 3, we state and prove
% the enumeration result in the hyperelliptic case.

%====================
\subsection*{Acknowledgments}
%====================
%The authors thank the referees for their careful reading and for helpful suggestions and comments.
%This work was supported by JSPS Grant-in-Aid for Scientific
%Research (C) 17K05196, and JSPS Grant-in-Aid for Young Scientists 20K14301.
This work was supported by JSPS Grant-in-Aid for Young Scientists 20K14301 and JSPS Grant-in-Aid for Scientific Research (C) 21K03159.
\section{Preliminaries}\label{sec:pre}
In this section, we collect some results
on counting rational points on a curve of genus $5$
with the classification of curves of genus $5$.

\subsection{Curves of genus 5}
There are three kinds of curves of genus $5$;
hyperelliptic, trigonal and the other case.
The main reference is \cite[Chap.~IV, Ex.~5.5]{Har}.

A hyperelliptic curve of genus $g$ over $K$ of characteristic $\ne 2$ is realized
as the desingularization of the projective closure of
\[
y^2 = f(x)
\]
for an $f(x) \in K[x]$ of degree $2g+2$.
Even if $\deg f = 2g+1$, we have a hyperelliptic curve of genus $g$,
but such a curve has an extra condition: the ramified points of the hyperelliptic fibration $C\to {\mathbb P}^1$ contains a $K$-rational point.

A trigonal curve $C$ is a curve admitting a morphism $\pi: C \to {\mathbb P}^1$  of degree $3$. If the genus of $C$ is $5$,
then the morphism $\pi$ is unique and the linear system $\omega_C\otimes {\mathcal L}^{-1}$ defines a morphism $\rho: C \to {\mathbb P}^2$,
where $\omega_C$ is the canonical sheaf and ${\mathcal L} =\pi^* O_{{\mathbb P}^1}(1)$. The image of $\rho$ turns out to be a quintic in ${\mathbb P}^2$
with single singular point, which is a $K$-rational point and is a node or a cusp. Thus, any trigonal curve of genus $5$ is realized as the desingularization of such a quintic in $\mathbb P^2$, see \cite[Subsection 2.1]{trigonal} for more details.

Finally it is known that any nonhyperelliptic and nontrigonal curve $C$
of genus $5$ is realized as a complete intersection of three quadrics in ${\mathbb P}^3$ \cite[Chap.~IV, Ex.~5.5]{Har}. 
But we need so many indeterminates to parametrize three quadrics.
To overcome the difficulty, in \cite{KH-genus5generic} the authors introduced
a sextic model of $\mathbb P^2$ constructed from $C$ and two distinct points on $C$, which enables us to parametrize these curves very efficiently.

\subsection{Rational points on hyperelliptic curves}\label{subsec:hyp}
The result in this subsection holds for arbitrary genus $g \ge 1$
and arbitrary characteristic $\ne 2$.
We do not require that the genus of $C$ is of genus $5$.

Consider a hyperelliptic curve $C$. Note that $C$ is the desingularization of
the projective closure of an affine model
\[
C^{\text{aff}}:\quad c y^2 = f(x)
\]
with a monic $f(x)\in K[x]$ of degree $2g+2$ and $c\in K^\times$.
We are concerned with the number of rational points on $C$.
\begin{lem}\label{lem:hyp_rat}
We have
\[
\#C(K) =\begin{cases}
\# C^{\text{\rm aff}}(K) + 2 & \text{if } c \in (K^\times)^2,\\
\# C^{\text{\rm aff}}(K) & \text{if } c \not\in (K^\times)^2 .
\end{cases}
\]
\end{lem}
\begin{proof}
The projective closure of $C^{\text{aff}}$ is
\[
cy^2 z^{2g} = z^{2g+2}f(x/z) = \sum_{i=0}^{2g+2} a_i x^iz^{2g+2-i}
\]
with $a_{2g+2}=1$. The desingularization of this around the infinity $(0,1,0)$ is given as follows. Consider the variables $X:=x/y$ and $Z:=z/y$.
The desingularization around $(X,Z)=(0,0)$ of
\[
cZ^{2g} = \sum_{i=0}^{2g+2} a_i X^iZ^{2g+2-i}
\]
is given by the equations
\[
u^g = vw,\quad c w^2 = \sum_{i=0}^{2g+2} a_i u^{2g+2-i}
\]
with new varibles $u,v,w$. Here the desingularization morphism is given by
$X=v$ and $Z=vu$,
whose inverse outside the origin is given by $u=Z/X$, $v=X$ and $w=Z^g/X^{g+1}$ . The points on the desingularization over $(X,Z)=(0,0)$ are $(u,v,w)=(0,0,\pm\sqrt{1/c})$. Thus we have the lemma.
\end{proof}

\subsection{Rational points on trigonal curves of genus $5$}\label{subsec:tri}
Let $C$ be a trigonal curve of genus $5$ over $K$.
Then $C$ is the desingularization of a quintic in $\mathbb P^2$ of the form
\[
F = q(x,y)z^3 + \varphi(x,y,z)
\]
with $q(x,y)$ is a quadratic form in $x,y$ over $K$
and $\varphi(x,y,z)\in K[x,y,z]$ is of degree 3
and of $z$-degree $\le 2$.
By the classification of quadratic forms,
$q(x,y)$ is isomorphic to $xy$, $x^2-\varepsilon y^2$ with $\varepsilon\not\in (K^\times)^2$ or $x^2$. The singular point in each case
is called {\it a split node}, {\it a non-split node} and {\it a cusp} over $K$.
Looking at the desingularization of the singular point, we have
\begin{lem}\label{lem:tri_rat}
Let $L$ be a finite field containing $K$. Then
\[
\#C(L) = \begin{cases}
\#V(F)(L) + 1 & \text{if the singular point is a split node over } L,\\
\#V(F)(L) - 1 & \text{if the singular point is a non-split node over } L,\\
\#V(F)(L) & \text{if the singular point is a cusp over } L.
\end{cases}
\]
\end{lem}

%========================================
\section{Reduction of defining equations}\label{sec:red}
%========================================

This section gives reductions of defining equations for hyperelliptic curves and genus-$5$ trigonal curves.
The reductions provided in this section can reduce the dimension
of our search space as much as possible.

%===========================================
\subsection{Reduction for hyperelliptic curves}\label{subsec:ReductionHyper}
%===========================================

The result in this subsection holds for arbitrary genus $g \ge 1$ and arbitrary characteristic $\ne 2$.
Let $K$ be a perfect field of odd characteristic $p > 2$.
In the below lemma (Lemma \ref{ReductionHyperelliptic}), we give a reduction of a defining equation for a hyperelliptic curve over $K$.
Note that the case (1) was already proved in \cite[Lemma 2]{KH18}, and the proof below is the same as that of \cite[Lemma 2]{KH18}.
By a slight modification of the proof of \cite[Lemma 2]{KH18}, we also prove a similar argument for the case (2):

% \begin{lem}[\cite{KH18}, Lemma 2]
% Assume that $p$ and $2g+2$ are coprime.
% Let $\epsilon \in K^\times \smallsetminus (K^\times)^2$.
% Any hyperelliptic curve $C$ of genus $g$ over $K$ is the desingularization of the homogenization of
% \[
% c y^2 = x^{2g+2} + b x^{2g} + a_{2g-1}x^{2g-1} + \cdots + a_1 x + a_0
% \]
% for $a_i \in K$ with $0 \leq i \leq 2g-1$ where $b= 0, 1,\epsilon$ and $c=1,\epsilon$.
% \end{lem}

\begin{lem}\label{ReductionHyperelliptic}
% Assume that $p$ divides $2g+2$.
Let $\epsilon \in K^\times \smallsetminus (K^\times)^2$.
Any hyperelliptic curve $H$ of genus $g$ over $K$ is the desingularization of the homogenization of the following affine model:
\begin{enumerate}
\item[\rm (1)] If $p$ and $2g+2$ are coprime,
\begin{equation}\label{eq:eq:genus5Hyp1}
c y^2 = x^{2g+2} + b x^{2g} + a_{2g-1}x^{2g-1} + \cdots + a_1 x + a_0
\end{equation}
with $b \in \{0, 1 , \epsilon\}$, $c\in \{1,\epsilon\}$ and $a_i \in K$ for $0 \leq i \leq 2g-1$.
\item[\rm (2)] If $p$ divides $2g+2$,
\begin{equation}\label{eq:genus5Hyp}
c y^2 = x^{2g+2} + b_1 x^{2g+1} + b_2 x^{2g} + a_{2g-1}x^{2g-1} + \cdots + a_1 x + a_0
\end{equation}
with $(b_1,b_2) \in \{(1,0), (0,0), (0,1), (0,\epsilon)\}$,
$c \in \{1,\epsilon\}$ and $a_i \in K$ for $0 \leq i \leq 2g-1$.
\end{enumerate}
% for $a_i \in K$ with $0 \leq i \leq 2g-1$.
% , where $(b_1,b_2)= (1,0), (0,0), (0,1), (0,\epsilon)$ and $c=1,\epsilon$.
\end{lem}
\begin{proof}
A hyperelliptic curve $H$ over $K$ is realized as $y^2 = f(x)$ for a polynomial $f(x)$ of degree $2g+2$ over $K$.
This can be expressed as $c y^2 = h_0(x)$ for $c\in K^\times$ and for a monic polynomial $h_0(x)$ of degree $2g+2$ over $K$.
Considering the transformation $(x,y) \mapsto (x,\alpha y)$ for some $\alpha\in K^\times$, one may assume $c=1$ or $\epsilon$.
We write $h_0(x) = x^{2g+2} + a_{2g+1}x^{2g+1} + \cdots + a_0$ with $a_i \in K$ for $0 \leq i \leq 2g+1$.
Considering $x \to x + \beta$, we can transform $h_0(x)$ to a polynomial
\begin{eqnarray}
h_1(x) &= & \sum_{j=0}^{2g+2} \binom{2g+2}{j} \beta^{j} x^{2g+2-j} + \sum_{i=0}^{2g+1} a_{i} \sum_{j=0}^{i} \binom{i}{j} \beta^{j} x^{i-j}  \nonumber \\
& = & x^{2g+2} + \left( \binom{2g+2}{1} \beta + a_{2g+1} \right) x^{2g+1} + \left( \binom{2 g + 2}{2} {\beta}^2 + a_{2g+1} \binom{2g+1}{1} \beta + a_{2g} \right) x^{2g} \nonumber \\
& & + \; (\mbox{lower terms in $x$}). \nonumber
% \begin{split}
% & x^{2g+2} + \left( \binom{2g+2}{1} \beta + a_{2g+1} \right) x^{2g+1} \\
% & + \left( \binom{2 g + 2}{2} {\beta}^2 + a_{2g+1} \binom{2g+1}{1} \beta + a_{2g} \right) x^{2g} + (\mbox{lower terms in $x$}).
% \end{split}
% \label{eq:hyp}
\end{eqnarray}
The transformation $(x,y) \mapsto (\gamma x, \gamma^{g+1}y)$ for some $\gamma \in K^\times$ and the multiplication by $\gamma^{-(2g+2)}$ to the whole of $c y^2 = h_1(x)$, we may assume $H$ is defined by an equation of the form
\begin{equation}
\begin{split}
c y^2 &= x^{2g+2} + \left( \binom{2g+2}{1} \beta + a_{2g+1} \right) \gamma^{-1} x^{2g+1} \\
& \quad + \left( \binom{2 g + 2}{2} {\beta}^2 + a_{2g+1} \binom{2g+1}{1} \beta + a_{2g} \right) \gamma^{-2} x^{2g} + \; (\mbox{lower terms in $x$}). \label{eq:hyp2}
\end{split}
\end{equation}
%\begin{enumerate}

%\item 
(1) Assume that $p$ and $2g+2$ are coprime.
In this case, we may assume that the $x^{2g+1}$-coefficient in \eqref{eq:hyp2} is zero, by putting $\beta := - a_{2g+1}/(2g+2)$.
For each value of $\beta$, there exists an element $\gamma \in K^{\times}$ such that the $x^{2g}$-coefficient in \eqref{eq:hyp2} becomes either of $0$, $1$ and $\epsilon$.

%\item 
(2) Assume that $p$ divides $2g+2$.
Since $\binom{2g+2}{1} \equiv 0 \bmod{p}$ and $\binom{2g+2}{2} \equiv 0 \bmod{p}$, we have
% \[
% h_1(x) = x^{2g+2} + a_{2g+1} x^{2g+1} + \left( a_{2g+1} \beta + a_{2g} \right) x^{2g}  + (\mbox{lower terms in $x$}) .
% \]
% The transformation $(x,y) \mapsto (\gamma x, \gamma^{g+1}y)$ for some $\gamma \in K^\times$ and the multiplication by $\gamma^{-(2g+2)}$ to the whole of $c y^2 = h_1(x)$, we may assume $C$ is defined by an equation of the form
\begin{equation}
c y^2 = x^{2g+2} + a_{2g+1} \gamma^{-1} x^{2g+1} + \left( a_{2g+1} \beta + a_{2g} \right) \gamma^{-2} x^{2g} + (\mbox{lower terms in $x$}). \label{eq:hyp1}
\end{equation}
If $a_{2g+1} \neq 0$, we may assume that the coefficient of $x^{2g+1}$ (resp.\ $x^{2g}$) in \eqref{eq:hyp1} is $1$ (resp.\ $0$), by putting $\gamma := a_{2g+1}$ and $\beta = - a_{2g} / a_{2g+1}$.
Otherwise we may assume that the $x^{2g}$-coefficient in \eqref{eq:hyp1} is either of $0$, $1$ and $\epsilon$.
%\end{enumerate}
\end{proof}

%============================================================================
\subsection{Reduction for trigonal curves of genus five in characteristic three}\label{subsec:ReductionTrigonal}
%============================================================================

Let $K$ be a finite field of characteristic $p=3$.
Throughout this section, let $q$ be the cardinality of $K$, $\zeta$ a primitive element of $K^\times$, and $\epsilon$ an element of $K^\times \smallsetminus (K^\times)^2$.
As quintic models of trigonal curves of genus $5$ over $K$, we have the following three types (cf.\ \cite[Subsection 2.1]{trigonal}):
\begin{enumerate}
\item[] {\bf (Split node case)} $C'=V(F)$ for some $F=xyz^3+f$;
\item[] {\bf (Non-split node case)} $C'=V(F)$ for some $F=(x^2-\epsilon y^2)z^3+f$;
\item[] {\bf (Cusp case)} $C'=V(F)$ for some $F=x^2z^3+f$,
\end{enumerate}
where $f$ is the sum of monomial terms which can not be divided by $z^3$.
Note that in the cusp case, the $y^3z^2$-coefficient of $f$ has to be non-zero.

In the following, we shall give reduced forms of $F$ in each of the above three cases. In \cite[Section 3]{trigonal} we considered similar reductions
but those do not work for $p = 3$.

%============================
%\subsection{Split node case}
%============================

\begin{prop}[Split node case]\label{ReductionSplitNode}
Any genus-five trigonal curve over $K$ of split node type
has a quintic model in $\mathbb{P}^2$ of the form
\begin{equation}\label{eq:SplitNodeQuinticGeneral}
\begin{split}
%F_0 &= 
&xyz^3 + (a_1 x^3 + a_2 x^2 y + a_3 x y^2 + a_4 y^3) z^2
+ (a_5 x^4 + a_6 x^3 y + a_7 x^2 y^2 + a_8 x y^3 + a_9 y^4) z\\
&\quad + a_{10} x^5 + a_{11} x^4 y + a_{12} x^3 y^2 + a_{13} x^2 y^3 + a_{14} x y^4 + a_{15} y^5
\end{split}
\end{equation}
for $a_i\in K$, where either of the following {\rm (1)} -- {\rm (5)} hold.
\begin{enumerate}
\item[\rm (1)] $a_1=1$, $a_2\in\{0,1\}$ and $a_5 = a_6 =0$;
\item[\rm (2)] $a_1=0$, $a_2=1$, $a_3\in\{0,1\}$ and $a_6=a_7=0$;
\item[\rm (3)] $a_1=a_2=0$, $a_3=1$, $a_4\in\{0,1\}$ and $a_7=a_8=0$;
\item[\rm (4)] $a_1=a_2=a_3=0$, $a_4=1$ and $a_8=a_9=0$;
\item[\rm (5)] $a_1=a_2=a_3=a_4=0$, $a_6 \in \{0,1,\zeta\}$ and $a_7\in\{0,1\}$.\end{enumerate}
\end{prop}

\begin{proof}
%Let $F_0$ be a quintic, say
%\begin{equation}\label{eq:SplitNodeQuinticGeneral}
%\begin{split}
%F_0 &= xyz^3 + (a_1 x^3 + a_2 x^2 y + a_3 x y^2 + a_4 y^3) z^2
%+ (a_5 x^4 + a_6 x^3 y + a_7 x^2 y^2 + a_8 x y^3 + a_9 y^4) z\\
%&\quad + a_{10} x^5 + a_{11} x^4 y + a_{12} x^3 y^2 + a_{13} x^2 y^3 + a_{14} x y^4 + a_{15} y^5.
%\end{split}
%\end{equation}
Note that \eqref{eq:SplitNodeQuinticGeneral} is the general form of the quintic
in the split node case.
Considering $z \to z + \alpha x + \beta y$, we can transform the quintic \eqref{eq:SplitNodeQuinticGeneral} to a quintic of the form:
\begin{equation}
\begin{split}
F &= xyz^3 + (a_1 x^3 + a_2 x^2 y + a_3 x y^2 + a_4 y^3) z^2 
+ f_1 z + f_0,
\end{split}
\end{equation}
where
\begin{equation}\label{eq:f1}
\begin{split}
f_1 &= (2 \alpha a_1 + a_5) x^4 + (2 \beta a_1 + 2 \alpha a_2 + a_6) x^3 y 
+ (2 \beta a_2 + 2 \alpha a_3 + a_7) x^2 y^2 \\
& \quad + (2 \beta a_3 + 2 \alpha a_4 + a_8) x y^3 + (2 \beta a_4 + a_9) y^4
\end{split}
\end{equation}
and $f_0$ is a quintic form in $x$ and $y$.
%\begin{enumerate}

%\item 
(1) ($a_1 \neq 0$) Putting $\alpha := - a_5 / (2 a_1)$ and $\beta := - (2 \alpha a_2 + a_6)/ (2 a_1)$, we may assume $a_5 = a_6 = 0$ in \eqref{eq:SplitNodeQuinticGeneral}.
Considering $(x,y) \to (\gamma x,\delta y)$ and the multiplication by $(\gamma \delta)^{-1}$, the coefficients $a_1$ and $a_2$ are transformed into $\gamma^2 \delta^{-1} a_1$ and $\gamma a_2$ respectively. 
Putting $\gamma := a_2^{-1}$ (resp.\ $\gamma := 1$) if $a_2 \neq 0$ (resp.\ $a_2 = 0$) and $\delta := \gamma^2 a_1$, we may assume that the coefficients of $x^3 z^2$ and $x^2 y z^2$ are $1$ and $0,1$ respectively.

%\item
(2) ($a_1 = 0$ and $a_2 \neq 0$) Putting $\alpha := - a_6 / (2 a_2)$ and $\beta := - (2 \alpha a_3 + a_7)/ (2 a_2)$, we may assume $a_6 = a_7 = 0$ in \eqref{eq:SplitNodeQuinticGeneral}.
Considering $(x,y) \to (\gamma x,\delta y)$ and the multiplication by $(\gamma \delta)^{-1}$, the coefficients $a_2$ and $a_3$ are transformed into $\gamma a_2$ and $\delta a_3$ respectively.
Putting $\gamma := a_2^{-1}$ and $\delta := a_3^{-1}$ (resp.\ $\delta := 1$) if $a_3 \neq 0$ (resp.\ $a_3 = 0$), we may assume that the coefficients of $x^2 y z^2$ and $x y^2 z^2$ are $1$ and $0,1$ respectively.

%\item
(3) ($a_1 = a_2 = 0$ and $a_3 \neq 0$) Putting $\alpha := - a_7 / (2 a_3)$ and $\beta := - (2 \alpha a_4 + a_8)/ (2 a_3)$, we may assume $a_7 = a_8 = 0$ in \eqref{eq:SplitNodeQuinticGeneral}.
Considering $(x,y) \to (\gamma x,\delta y)$ and the multiplication by $(\gamma \delta)^{-1}$, the coefficients $a_3$ and $a_4$ are transformed into $\delta a_3$ and $\gamma^{-1}\delta^2 a_4$ respectively.
Putting $\delta := a_3^{-1}$ and $\gamma := \delta^2 a_4$ (resp.\ $\gamma := 1$) if $a_4 \neq 0$ (resp.\ $a_4 = 0$), we may assume that the coefficients of $x y^2 z^2$ and $y^3 z^2$ are $1$ and $0,1$ respectively.

%\item
(4) ($a_1 = a_2 = a_3 = 0$ and $a_4 \neq 0$) Putting $\alpha := - a_8 / (2 a_4)$ and $\beta := - a_9/ (2 a_4)$, we may assume $a_8 = a_9 = 0$ in \eqref{eq:SplitNodeQuinticGeneral}.
Considering $(x,y) \to (\gamma x,\delta y)$ and the multiplication by $(\gamma \delta)^{-1}$, the coefficient $a_4$ is transformed into $\gamma^{-1}\delta^2 a_4$.
Putting $\gamma := a_4$ and $\delta := 1$, we may assume that the coefficient of $y^3 z^2$ is $1$.

%\item
(5) ($a_1 = a_2 = a_3 = a_4 = 0$) Considering $(x,y) \to (\gamma x,\delta y)$ and the multiplication by $(\gamma \delta)^{-1}$, the coefficients $a_6$ and $a_7$ are transformed into $\gamma^2 a_6$ and $\gamma \delta a_7$ respectively.
If $a_6 = 0$, put $\gamma := 1$ and $\delta := a_7^{-1}$ if $a_7 \neq 0$.
If $a_6 \neq 0$, write $a_6 = \zeta^{2 i + j}$ for $i \in \{0, \ldots , (q-3)/2 \}$ and $j \in \{0, 1 \}$, and put $\gamma := \zeta^{-i}$ and $\delta := (\gamma a_7)^{-1}$ if $a_7 \neq 0$.
%\end{enumerate}
%
% whose coefficients of $x^2yz^2$ and $xy^2z^2$ are zero.
% One may write
% \begin{equation*}
% \begin{split}
% F &= xyz^3 + (b_0 x^3 + b_1 y^3) z^2
% + (a_1 x^4 + a_2 x^3 y + a_3 x^2 y^2 + a_4 x y^3 + a_5 y^4) z\\
% &\quad + a_6 x^5 + a_7 x^4 y + a_8 x^3 y^2 + a_9 x^2 y^3 + a_{10} x y^4 + a_{11} y^5.
% \end{split}
% \end{equation*}
%
% If $(b_0,b_1)\ne (0,0)$, then considering
% the exchange of $x$ and $y$
% we may assume $b_0\ne 0$.
% Consider $(x,y) \to (\alpha x,\beta y)$ and the multiplication by $(\alpha\beta)^{-1}$,
% the coefficients $b_0$ and $b_1$ are transformed into $b_0\alpha^2\beta^{-1}$ and $b_1\alpha^{-1}\beta^2$ respectively.
% Set $\beta := b_0\alpha^2$; then $b_0$ and $b_1$
% become $1$ and $b_0^2b_1\alpha^3$ respectively.
% Thus we may assume that $b_0 = 1$ and
% $b_1$ is $0$ or a representative of
% an element of $K^\times/(K^\times)^3$, i.e.,
% $(b_0,b_1) = (1,0), (1,1), (1,\zeta), (1,\zeta^2)$.
% Considering the exchange of $x$ and $y$ again, one may reduce to
% $(b_0,b_1) = (1,0), (1,1), (1,\zeta)$.
%
% If $(b_0,b_1) = (0,0)$, then there exist elements $\alpha, \beta$ of $K^\times$ such that the transformation $(x,y)\mapsto (\alpha x, \beta y)$
% and the multiplication by $(\alpha\beta)^{-1}$ to the whole
% make $F$ the form of (2).
\end{proof}

%===============================
%\subsection{Non-split node case}
%===============================
%\noindent\underline{Non-split node case}: 
% Let $\epsilon$ be an element of $K^\times \smallsetminus (K^\times)^2$.

%\begin{lem}\label{RepresentationRotationGroup}
%Consider the natural representation of $\tilde \gC$
%on the space $V$ of cubics in $y,z$ over $K$.
%\begin{enumerate}
%\item[\rm (1)]
%$V$ is the direct sum of two subrepresentations $V_1:=\langle y(y^2-\epsilon z^%2), z(y^2-\epsilon z^2)\rangle $
%and $V_2:=\langle y(y^2+3\epsilon z^2), z(3y^2+\epsilon z^2)\rangle$.
%%\item[\rm (2)] $\tilde{\rm C}$ is isomorphic to ${\tilde K}^\times$
%%with $\tilde K=K(\sqrt{-(-\epsilon)})$.
%%The representation $V_1$ of $\tilde{\rm C}$
%%is equivalent to
%%the representation on $\F_{q^2}$ of $\F_{q^2}^\times$ defined by
%%$v\mapsto k^{q+2}v$ for $k \in \F_{q^2}^\times$ and $v\in \F_{q^2}$.
%\item[\rm (2)]
%There are four $\tilde \gC$-orbits in $V_1$, which are the orbits
%of $\delta y(y^2-\epsilon z^2)$ with $\delta \in \{0\} \cup K^\times/(K^\times)%^3$.
%\end{enumerate}
%\end{lem}

\begin{prop}[Non-split node case]\label{ReductionNonSplitNode}
Any genus-five trigonal curve over $K$ of non-split node type
has a quintic model in $\mathbb{P}^2$ of the form
\begin{enumerate}
\item[\rm (1)] 
for $a_i\in K$,
\begin{equation}\label{NonSplitNodeReducedEq1}
\begin{split}
F &= (x^2-\epsilon y^2)z^3 + 
\left\{a_1 x(x^2-\epsilon y^2) +  
a_2x^3 + y^3\right\} z^2 + (a_3 x^4 + a_4 x^3 y + a_5 x^2 y^2) z \\
&\quad + a_6 x^5 + a_7 x^4 y + a_8 x^3 y^2 + a_9 x^2 y^3 + a_{10} x y^4 + a_{11} y^5.
\end{split}
\end{equation}
\item[\rm (2)]for $a_i\in K$,
\begin{equation}\label{NonSplitNodeReducedEq2}
\begin{split}
F &= (x^2-\epsilon y^2)z^3 + 
\left\{x(x^2-\epsilon y^2) +  
a_1x^3 \right\} z^2 + (a_2 x^4 + a_3 x^3 y + a_4 y^4) z \\
&\quad + a_5 x^5 + a_6 x^4 y + a_7 x^3 y^2 + a_8 x^2 y^3 + a_{9} x y^4 + a_{10} y^5.
\end{split}
\end{equation}
\item[\rm (3)]
for $a_i\in K$,
\begin{equation}\label{NonSplitNodeReducedEq3}
\begin{split}
F &= (x^2-\epsilon y^2)z^3 + x^3 z^2 + (a_1 x^2 y^2 + a_2 x y^3 + a_3 y^4) z \\
&\quad + a_4 x^5 + a_5 x^4 y + a_6 x^3 y^2 + a_7 x^2 y^3 + a_8 x y^4 + a_9 y^5.
\end{split}
\end{equation}
\item[\rm (4)] for $a_i\in K$,
\begin{equation}\label{NonSplitNodeReducedEq4}
\begin{split}
F &= (x^2-\epsilon y^2)z^3  + (a_1 x^4 + a_2 x^3 y + a_3 x^2 y^2 + a_4 x y^3 + a_5 y^4) z \\
&\quad + a_6 x^5 + a_7 x^4 y + a_8 x^3 y^2 + a_9 x^2 y^3 + a_{10} x y^4 + a_{11} y^5,
\end{split}
\end{equation}
where $(a_6,\ldots,a_{11})$ is either of $(0,\ldots,0,1,a_{i+1},\ldots, a_{11})$
for $i=6,\ldots,11$.
\end{enumerate}
\end{prop}
\begin{proof}
%Let $K=\F_q$.
Let $V$ be the $K$-vector space consisting of cubic forms in $x,y$ over $K$.
As seen in \cite[Lemma 4.1.1]{KH16},
the representation $V$ of $\tilde{\rm C}$
defined by $\gamma x = rx+\epsilon s y$ and $\gamma y = s x +r y$
for
\[
\gamma = \begin{pmatrix}r & \epsilon s\\ s & r\end{pmatrix}\in\tilde{\rm C}
\]
is decomposed as $V_1\oplus V_2$, where
$V_1 = \langle x(x^2-\epsilon y^2), y(x^2-\epsilon y^2)\rangle$ and
$V_2 = \langle x(x^2+3\epsilon y^2), y(3x^2 + \epsilon y^2)\rangle = \langle x^3, y^3\rangle$.
We write
\begin{equation}
\begin{split}
F &= (x^2-\epsilon y^2)z^3 + 
\left\{c_1 x(x^2-\epsilon y^2) + c_2 y(x^2-\epsilon y^2) +  
b_1x^3 + b_2 y^3\right\} z^2 \\
&\quad + (a_1 x^4 + a_2 x^3 y + a_3 x^2 y^2 + a_4 x y^3 + a_5 y^4) z \\
&\quad + a_6 x^5 + a_7 x^4 y + a_8 x^3 y^2 + a_9 x^2 y^3 + a_{10} x y^4 + a_{11} y^5.
\end{split}
\end{equation}
An element $\gamma = \begin{pmatrix}r & \epsilon s\\ s & r\end{pmatrix}$ of $\tilde{\rm C}$
sends $c_1 x(x^2-\epsilon y^2)  + c_2 y(x^2-\epsilon y^2)$ to
\[
(r^2-\epsilon s^2) ((c_1r+c_2s) x(x^2-\epsilon y^2) + (c_1\epsilon s+c_2r) y(x^2-\epsilon y^2)).
\]
As there exists an $(r,s)\in K^2\setminus\{(0,0)\}$ so that $c_1\epsilon s+c_2r=0$,
we may assume that $c_2=0$:
\begin{equation}
\begin{split}
F &= (x^2-\epsilon y^2)z^3 + 
\left\{c_1 x(x^2-\epsilon y^2) +  
b_1x^3 + b_2 y^3\right\} z^2 \\
&\quad + (a_1 x^4 + a_2 x^3 y + a_3 x^2 y^2 + a_4 x y^3 + a_5 y^4) z \\
&\quad + a_6 x^5 + a_7 x^4 y + a_8 x^3 y^2 + a_9 x^2 y^3 + a_{10} x y^4 + a_{11} y^5.
\end{split}
\end{equation}
%\begin{enumerate}

%\item[(1)] 
(1) Assume $b_2\ne 0$.
We take the coordinate-change $(x,y,z)\mapsto (x,y,b_2 z)$ and multiply $F$ by $b_2^{-3}$;
then we may assume $b_2 = 1$. 
Considering the transformation sending $z$ to $z - (a_4/2+c_1\epsilon a_5/4) x - (a_5/2) y$,
we may eliminate the terms of $xy^3$
and $y^4$ from $F$, i.e., we may assume $(a_4,a_5)=(0,0)$:
\begin{equation}
\begin{split}
F &= (x^2-\epsilon y^2)z^3 + 
\left\{c_1 x(x^2-\epsilon y^2) +  
b_1x^3 + y^3\right\} z^2  + (a_1 x^4 + a_2 x^3 y + a_3 x^2 y^2) z \\
&\quad + a_6 x^5 + a_7 x^4 y + a_8 x^3 y^2 + a_9 x^2 y^3 + a_{10} x y^4 + a_{11} y^5.
\end{split}
\end{equation}

%\item[(2)]
(2) Assume $b_2=0$ and $c_1 \ne 0$. 
We take the coordinate-change $(x,y,z)\mapsto (x,y,c_1 z)$ and multiply $F$ by $c_1^{-3}$;
then we may assume $c_1 = 1$.
Considering the transformation sending $z$ to $z - (a_3/(-2\epsilon)) x - (a_4/(-2\epsilon)) y$,
we may eliminate the terms of $x^2y^2$
and $xy^3$ from $F$, i.e., we may assume $(a_3,a_4)=(0,0)$:
\begin{equation}
\begin{split}
F &= (x^2-\epsilon y^2)z^3 + 
\left\{x(x^2-\epsilon y^2) +  
b_1x^3 \right\} z^2 + (a_1 x^4 + a_2 x^3 y + a_5 y^4) z \\
&\quad + a_6 x^5 + a_7 x^4 y + a_8 x^3 y^2 + a_9 x^2 y^3 + a_{10} x y^4 + a_{11} y^5.
\end{split}
\end{equation}

%\item[(3)]
(3) Assume $b_2=0$ and $c_1 = 0$ and $b_1 \ne 0$.
Similarly $F$ is reduced to
\begin{equation}
\begin{split}
F &= (x^2-\epsilon y^2)z^3 + x^3 z^2 + (a_3 x^2 y^2 + a_4 x y^3 + a_5 y^4) z \\
&\quad + a_6 x^5 + a_7 x^4 y + a_8 x^3 y^2 + a_9 x^2 y^3 + a_{10} x y^4 + a_{11} y^5.
\end{split}
\end{equation}
%\end{enumerate}
\end{proof}

\begin{prop}[Cusp case]\label{ReductionCusp}
Any genus-five trigonal curve over $K$ of cusp type has a quintic model in $\mathbb{P}^2$ of the form 
\begin{equation}
\begin{split}
F & =  x^2 z^3 + (b_1 x^3 + b_2 x^2 y + a_1 x y^2 + a_2 y^3) z^2
+ (a_3 x^4 + a_4 x^3 y + a_5 x^2 y^2) z \\
& \quad + a_6 x^5 + a_7 x^4 y + a_8 x^3 y^2 + a_9 x^2 y^3 + a_{10} x y^4 + a_{11} y^5 \label{CuspReducedEq1}
\end{split}
\end{equation}
for $a_i\in K$ with $a_2 \neq 0$, where $b_1,b_2\in \{0, 1\}$.
\end{prop}

\begin{proof}
Recall from the paragraph at the beginning of this section that any trigonal curve of genus $5$ over $K$ of cusp type has a quintic model in $\mathbb{P}^2$ of the form  
\begin{equation}\label{eq:CuspQuinticGeneral}
\begin{split}
F_0 &= x^2 z^3 + (a_1 x^3 + a_2 x^2 y + a_3 x y^2 + a_4 y^3) z^2
+ (a_5 x^4 + a_6 x^3 y + a_7 x^2 y^2 + a_8 x y^3 + a_9 y^4) z\\
&\quad + a_{10} x^5 + a_{11} x^4 y + a_{12} x^3 y^2 + a_{13} x^2 y^3 + a_{14} x y^4 + a_{15} y^5,
\end{split}
\end{equation}
where $a_i \in K$ with $a_4 \neq 0$.
Considering $z \to z + \alpha x + \beta y$, we can transform the quintic \eqref{eq:CuspQuinticGeneral} to a quintic of the form:
\begin{equation}
\begin{split}
F &= x^2 z^3 + (a_1 x^3 + a_2 x^2 y + a_3 x y^2 + a_4 y^3) z^2 
+ f_1 z + f_0,
\end{split}
\end{equation}
where $f_1$ is the same as in \eqref{eq:f1}, and where $f_0$ is a quintic form in $x$ and $y$.
Putting $\beta := - a_9 / (2 a_4)$ and $\alpha := - (2 \beta a_3 + a_8)/(2a_4)$, we may assume $a_8 = a_9 = 0$ in \eqref{eq:CuspQuinticGeneral}.
Considering $(x,y) \to (\gamma x,\delta y)$ and the multiplication by $\gamma^{-2}$, the coefficients $a_1$ and $a_2$ are transformed into $\gamma a_1$ and $\delta a_2$ respectively.
Thus, we may assume that the coefficients of $x^3 z^2$ and $x^2 y z^2$ are $0$ or $1$.
\end{proof}

%======================================
\section{Main results and their proofs}\label{sec:main}
%======================================

In this section, we prove the main theorems (Theorems \ref{thm:main2} and \ref{thm:main}) with the help of computer calculations, which were done over the computer algebra system Magma V2.26-10~\cite{Magma}, \cite{MagmaHP} on a computer with macOS Monterey 12.0.1, at 2.6 GHz CPU 6 Core (Intel Core i7) and 16GB memory.
The source codes and the log files are summarized at \cite{HPkudo}.

%===============================
\subsection{Hyperelliptic case}\label{subsec:MainHyp}
%===============================

\paragraph{\it Proof of Theorem \ref{thm:main2}.}
It follows from Lemma \ref{ReductionHyperelliptic} that any hyperelliptic curve $H$ of genus $5$ over $\mathbb{F}_3$ is given by the desingularization of the homogenization of \eqref{eq:genus5Hyp} with $g=5$ and $\epsilon := 2 \in \mathbb{F}_3$.
In order to determine the maximal number of $\mathbb{F}_{9}$-rational points of $H$, we conducted an exhaustive search on all possible unknown coefficients in \eqref{eq:genus5Hyp} over Magma.
More precisely, we implemented the following procedure (dividing into two steps), and executed it for $q=9$:
\begin{proc}\label{proc:hyp}
Putting $N :=0$, $\mathcal{H} := \emptyset$ and $\mathrm{Isom}(\mathcal{H}) := \emptyset$, proceed with the following steps:
\begin{description}
\item[{\it Step 1.}] For each of possible $(c,b_1,b_2,a_0, \ldots , a_9)$:
\begin{enumerate}
\item Check whether the right hand side of \eqref{eq:genus5Hyp} (we denote it by $f(x)$ here) is square free or not, by computing $\mathrm{gcd}(f,f')$.
\item If $\mathrm{gcd}(f,f')=1$, i.e., $f(x)$ is square free, compute the number $\# H (\mathbb{F}_q)$ of $\mathbb{F}_q$-rational points of $H$ defined by \eqref{eq:genus5Hyp}, with a formula given in Lemma \ref{lem:hyp_rat}.
\item If $\# H (\mathbb{F}_q) > N$, put $N:= \# H (\mathbb{F}_q)$ and $\mathcal{H} := \{ H \}$.
If $\# H (\mathbb{F}_q) = N$, put $\mathcal{H}:= \mathcal{H} \cup \{ H \}$.
\end{enumerate}
Write $\mathcal{H} := \{ H_1, \ldots , H_s \}$ with $s := \# \mathcal{H}$.
\item[{\it Step 2.}] For $i$ from $1$ to $s$ by $1$:
\begin{itemize}
\item Test whether $H_i$ is isomorphic to $H_j$ over $\mathbb{F}_q$ for some $j < i$.
If $H_i$ is not isomorphic to $H_j$ for any $j < i$, then put $\mathrm{Isom}(\mathcal{H}) := \mathrm{Isom}(\mathcal{H}) \cup \{ H_i \}$.
\end{itemize} 
Output $N$ and $\mathrm{Isom}(\mathcal{H}) $.
\end{description}
\end{proc}
The source codes for Steps 1 and 2 and their log files are available at \cite{HPkudo}, for example, the code for Step 1 is named \texttt{hyper\_g5q9\_step1.txt}.
Note that Step 2 was done with the output of Step 1.
For Step 1 (2), we used Magma's built-in function \texttt{Variety} to compute $\# H^{\text{\rm aff}}(\mathbb{F}_q)$.
Isomorphism testing in Step 2 was done with an algorithm provided in \cite{KH18} (see Remark \ref{rem:hyp} below for a brief review).
It took about only 50 seconds (resp.\ 30.6 hours) to execute the code for Step 1 (resp.\ Step 2).

The output of Step 1 (\texttt{log\_hyper\_g5q9\_step1.txt}) shows that the resulting $N$ is equal to $20$, i.e., the maximal number of $\mathbb{F}_9$-rational points of hyperelliptic curves defined by \eqref{eq:genus5Hyp} is $20$, and that the resulting $\mathcal{H}$ consists of $12048$ elements.
It follows from the output of Step 2 (\texttt{log\_hyper\_g5q9\_step2.txt}) that $\# \mathrm{Isom}(\mathcal{H}) = 573$, namely the collected $12048$ hyperelliptic curves are divided into $573$ $\mathbb{F}_9$-isomorphism classes.
We also computed the Weil polynomials of the $573$ curves by a method described in Remark \ref{rem:WeilPoly} below.
(At \cite{HPkudo}, the source code for computing the Weil polynomials and its log file are named \texttt{hyper\_g5q9\_WeilPoly.txt} and \texttt{log\_hyper\_g5q9\_WeilPoly.txt} respectively at \cite{HPkudo}, and the time it took to execute the code is about 830 seconds.)
\qed

\vskip\baselineskip

We also computed the number of $\mathbb{F}_3$-rational points of each hyperelliptic curve of genus $5$ over $\mathbb{F}_3$, and obtain the following:

\begin{prop}\label{thm:main3}
The maximal number of $\# H(\mathbb{F}_{3})$ of genus-five hyperelliptic curves $H$ over $\mathbb{F}_3$ is $8$.
Moreover, there are exactly $820$ $\mathbb{F}_{3}$-isomorphism classes of genus-five hyperelliptic curves $H$ over $\mathbb{F}_3$ with $8$ $\mathbb{F}_{3}$-rational points.
In Example \ref{ex:hyp} (2) below, examples of genus-five hyperelliptic curves $H$ over $\mathbb{F}_3$ with $\# H (\mathbb{F}_{3}) = 8$ will be given.
\end{prop}

\begin{proof}
Similarly to the proof of Theorem \ref{thm:main2}, we executed Procedure \ref{proc:hyp} for $q:=3$ over Magma.
All the source codes and the log files for this proof are available at \cite{HPkudo}, and it took about only 15 seconds (resp.\ 4.2 hours) to execute the code for Step 1 (resp.\ Step 2).
The output of Step 1 shows that $N = 8$ and $\# \mathcal{H} = 8293$, namely, the maximal value of $\# H (\mathbb{F}_3) $ is $8$, and there are $8293$ coefficient sequences $(c,b_1,b_2,a_0, \ldots , a_9)$ such that $H$ defined by \eqref{eq:genus5Hyp} satisfies $\# H (\mathbb{F}_3) = 8$.
From the output of Step 2, we have $\# \mathrm{Isom}(\mathcal{H}) = 820$, as desired.
\end{proof}

\begin{rem}\label{rem:hyp}
In the proofs of Theorem \ref{thm:main2} and Proposition \ref{thm:main3}, we used an algorithm given in \cite[Section 3.3]{KH18}, which tests whether two hyperelliptic curves are isomorphic or not over any field $k$ containing the defining finite field $K$.
Here we review the algorithms briefly:
Let $H_1 : c_1 y^2 = f_1(x)$ and $H_2 : c_2 y^2 = f_2(x)$ be hyperelliptic curves of genus $g$ with $c_1,c_2 \in K^{\times}$ and square-free polynomials $f_1$ and $f_2$ in $K[x]$ of degree $2 g + 2$.
For each $1 \leq i \leq 2$, let $F_i$ denote the homogenization of $c_i^{-1}f_i$ with respect to an extra variable $z$.
Then we have that $H_1 \cong H_2$ over $k$ if and only if there exist $h \in \mathrm{GL}_2 ( k )$ and $\lambda \in k^{\times}$ such that $h \cdot F_1 = \lambda^2 F_2$, see \cite[Lemma 1]{KH18}.
Regarding entries of $h$ and $\lambda$ as variables, we reduce the (non-)existence of such $h$ and $\lambda$ into that of a solution over $k$ of a multivariate system.
% The (non-)existence of such a solution is decided by computing a Gr\"{o}bner basis, see \cite[Section 3.3]{KH18} for more details.
The existence of a solution is equivalent to that for any monomial order, the (reduced) Gr\"{o}bner basis of the ideal defining the system includes $1$ as an element.
Thus, we can decide whether two hyperelliptic curves $H_1$ and $H_2$ given as above are isomorphic or not over $k$, by computing the reduced Gr\"{o}bner basis for a monomial order.
% see \cite[Section 3.3]{KH18} for more details including concrete algorithms to test the (non)-existence of such a solution.
\end{rem}

\begin{rem}\label{rem:WeilPoly}
As stated at the end of the proof of Theorem \ref{thm:main2}, we computed the Weil polynomial of each of the $2295$ curves.
A computational method which we adopt is as follows: 
Let $C$ be a nonsingular curve of genus $g$ over $\F_q$.
The Weil polynomial of $C$, say $W(t)=\prod_{i=0}^{2g} (t-\alpha_i)=\sum_{i=0}^{2g} a_i t^t$, is determined by the numbers $\# C(\F_{q^e})$ for $e=1, 2, \ldots, g$.
Indeed, the values $a_0, a_1, \ldots, a_g$
with $a_k = (-1)^k\sum_{i_1<t_2<\ldots<i_k}\alpha_{i_1}\alpha_{i_2}\cdots\alpha_{i_k}$
are determined by Newton's identities and $\sum_{i=0}^{2g} \alpha_i^e = 1 + q^e - \# C(\F_{q^e})$ for $e=1, 2, \ldots, g$. The remaining values $a_{g+1},\cdots, a_{2g}$
are determined by the formula $a_{2g-i}=q^{g-i}a_i$,
which follows from the functional equation of the congruent zeta function.
Our computation of Weil polynomials just used this method.
\end{rem}

\begin{ex}\label{ex:hyp}
\begin{enumerate}
\item The genus-five hyperelliptic curves over $\mathbb{F}_3$ defined by
\begin{eqnarray}
 y^2 &=& x^{12} + x^{11} + 2x^7 + x^5  + 2x + 1, \nonumber \\
 y^2 &=& x^{12} + x^{11} + 2x^4 + 2x^3 + 2 \nonumber
\end{eqnarray}
have $20$ $\mathbb{F}_9$-rational points, and their Weil polynomials are
\begin{eqnarray}
& & (t + 3)^2 (t^2 + 9)^2 (t^2 + 2 t + 9)^2, \nonumber \\
& & t^{10} + 10 t^9 + 51 t^8 + 212 t^7 + 837 t^6 + 2810 t^5 + 7533 t^4 + 17172 t^3 + 37179 t^2 + 65610 t + 59049, \nonumber
\end{eqnarray}
respectively.
\item The genus-five hyperelliptic curves over $\mathbb{F}_3$ defined by
\begin{eqnarray}
 y^2 &=& x^{12} + x^{11} + 2x^2 + 2x + 1, \nonumber \\
 y^2 &=& x^{12} + x^{11} + x^3 + 2x^2 + x + 1 \nonumber
\end{eqnarray}
have $8$ $\mathbb{F}_3$-rational points.
\end{enumerate}
\end{ex}

%=========================
\subsection{Trigonal case}\label{subsec:MainTri}
%=========================

% \begin{proof}
\paragraph{\it Proof of Theorem \ref{thm:main}.}
Recall from Subsection \ref{subsec:ReductionTrigonal} that there are three types of quintic models $V(F)$ of trigonal curves of genus $5$ over $\mathbb{F}_3$, with explicit forms of $F$ as follows:
{\bf (Split node case)}: \eqref{eq:SplitNodeQuinticGeneral} in Proposition \ref{ReductionSplitNode},
{\bf (Non-Split node case)}: \eqref{NonSplitNodeReducedEq1} -- \eqref{NonSplitNodeReducedEq4} in Proposition \ref{ReductionNonSplitNode}, and
{\bf (Cusp case)}: \eqref{CuspReducedEq1} in Proposition \ref{ReductionCusp}.
% : (Split node case), (Non-Split node case) and (Cusp case).
% In each case, we have explicit forms of $F$ as follows:
% \begin{itemize}
% \item (Split node case): \eqref{SplitNodeReducedEq1} -- \eqref{SplitNodeReducedEq5} in Proposition \ref{ReductionSplitNode}.
% \item (Non-Split node case): \eqref{NonSplitNodeReducedEq1} -- \eqref{NonSplitNodeReducedEq4} in Proposition \ref{ReductionNonSplitNode}.
% \item \eqref{CuspReducedEq1} in Proposition \ref{ReductionCusp}.
% \end{itemize}
For each form (\eqref{eq:SplitNodeQuinticGeneral}, \eqref{NonSplitNodeReducedEq1} -- \eqref{NonSplitNodeReducedEq4}, or \eqref{CuspReducedEq1}) of $F$, we conducted an exhaustive search on all possible unknown coefficients over Magma, similarly to the case of hyperelliptic curves (cf.\ Step 1 of Procedure \ref{proc:hyp} in the proof of Theorem \ref{thm:main2} given in Subsection \ref{subsec:MainHyp}).
% For collected quintic $V(F)$, we then classify the $\mathbb{F}_9$-isomorphism classes of their normalizations, by an algorithm provided in \cite{KH20}.
For the computation of the number of rational points, we used a formula given in Lemma \ref{lem:tri_rat}.
The source code and its log file are available at \cite{HPkudo}, and they are named \texttt{trigonal\_g5q9\_step1.txt} and \texttt{log\_trigonal\_g5q9\_step1.txt} respectively.
It took about 2.5 hours to execute the code.

We found from the output (\texttt{log\_trigonal\_g5q9\_step1.txt}) that the maximal number of $\mathbb{F}_9$-rational points of the normalization $C$ of such a quintic $V(F)$ is $30$.
In the following, we list quintic forms $F$ such that $\#C (\mathbb{F}_9)=30$:
\begin{itemize}
\item The quintic forms $F$ in split node case (of the form \eqref{eq:SplitNodeQuinticGeneral}) such that the normalization of $V(F) \subset \mathbb{P}^2$ has $30$ $\mathbb{F}_9$-rational points are the following:
\begin{eqnarray}
F_1 & = & x y z^3 + (x^3 + x y^2 + y^3) z^2 + x y^3 z + x^5 + x^3 y^2 + x^2 y^3 + y^5; \nonumber \\
F_2 & = & x y z^3 + (x^3 + 2 x y^2 + y^3) z^2 + (2 x^2 y^2 + 2 x y^3) z + x^5 + 2 x^4 y + x^2 y^3 + 2 x y^4 + y^5; \nonumber \\
F_3  & = & x y z^3 + (x^3 + x y^2 + 2 y^3) z^2 + (x^2 y^2 + x y^3 + 2 y^4) z + x^5 + x^3 y^2 + 2 x^2 y^3 + 2 x y^4 + y^5; \nonumber \\ 
F_4 & = & x y z^3 + (x^3 + x y^2 + y^3) z^2 + x y^3 z + x^5 + x^3 y^2 + 2 x^2 y^3 + 2 y^5; \nonumber \\
F_5 & = & x y z^3 + (x^3 + x y^2 + y^3) z^2 + (2 x^2 y^2 + x y^3 + y^4) z + x^5 + x^3 y^2 + x^2 y^3 + 2 x y^4 + 2 y^5; \nonumber \\
F_6 & = & x y z^3 + (x^3 + 2 x y^2 z^2 + 2 y^3) z^2 + (x^2 y^2 + 2 x y^3) z + x^5 + x^4 y + 2 x^2 y^3 + 2 x y^4 + 2 y^5; \nonumber \\
F_7 & = & x y z^3 + (x^3 + x^2 y + y^3) z^2 + (2 x^2 y^2 + 2 x y^3) z + x^5 + 2 x^4 y + 2 x^2 y^3 + y^5; \nonumber \\
F_8 & = & x y z^3 + (x^3 + x^2 y + 2 y^3) z^2 + (2 x^2 y^2 + 2 y^4) z + x^5 + x^4 y + x y^4 + y^5; \nonumber \\
F_9 & = & x y z^3 + (x^3 + x^2 y + y^3) z^2 + (2 x^2 y^2 + 2 y^4) z + 2 x^5 + x^3 y^2 + 2 x^2 y^3 + x y^4 + 2 y^5; \nonumber \\
F_{10} & = & x y z^3 + (x^3 + x^2 y + 2 x y^2 + 2 y^3) z^2 + 2 x^2 y^2 z + x^5 + x^4 y + 2 x y^4 + 2 y^5. \nonumber
\end{eqnarray}
% Note that the quintic forms $F_i$ for $1 \leq i \leq 10$ are of the same form \eqref{SplitNodeReducedEq1}.

\item The quintic forms $F$ in non-split node case (of the form \eqref{NonSplitNodeReducedEq1} -- \eqref{NonSplitNodeReducedEq4}) such that the normalization of $V(F) \subset \mathbb{P}^2$ has $30$ $\mathbb{F}_9$-rational points are the following:
\begin{eqnarray}
F_{11} & = & (x^2 - \epsilon y^2)z^3 + (x (x^2 - \epsilon y^2) + x^3 + y^3) z^2 + x^2 y^2 z + 2 x^5 + 2 x^4 y + x^3 y^2 + y^5; \nonumber \\
F_{12} & = & (x^2 - \epsilon y^2)z^3 + (2 x ( x^2 - \epsilon y^2) + 2 x^3 + y^3) z^2 x^2 y^2 z + x^5 + 2 x^4 y + 2 x^3 y^2 + y^5; \nonumber \\
F_{13} & = & (x^2 - \epsilon y^2)z^3 + (2 x (x^2 - \epsilon y^2) + x^3 + y^3) z^2 + (x^4 + x^3 y)z + x^5 + 2 x^2 y^3 + y^5; \nonumber \\ 
F_{14} & = & (x^2 - \epsilon y^2)z^3 + (x (x^2 - \epsilon y^2) + 2 x^3 + y^3) z^2 + (x^4 + 2 x^3 y) z + 2 x^5 + 2 x^2 y^3 + y^5; \nonumber \\
F_{15} & = & (x^2 - \epsilon y^2)z^3 + (x (x^2 - \epsilon y^2) + y^3) z^2 + (2 x^3 y + 2 x^2 y^2) z + x^5 + x^4 y + x^3 y^2 + 2 x^2 y^3 + y^5; \nonumber \\
F_{16} & = & (x^2 - \epsilon y^2)z^3 + (2 x (x^2 - \epsilon y^2) + y^3) z^2 + (x^3 y + 2 x^2 y^2) z + 2 x^5 + x^4 y + 2 x^3 y^2 + 2 x^2 y^3 + y^5; \nonumber \\
F_{17} & = & (x^2 - \epsilon y^2)z^3 + (x (x^2 - \epsilon y^2) + 2 x^3) z^2 + (x^4 + y^4) z + x^5,  \nonumber 
\end{eqnarray}
where $\epsilon = 2 \in \mathbb{F}_3^{\times} \smallsetminus (\mathbb{F}_3^{\times})^2$.
% Note that the quintic forms $F_i$ for $11 \leq i \leq 16$ are of the form \eqref{NonSplitNodeReducedEq1}, and $F_{17}$ is of the form \eqref{NonSplitNodeReducedEq2}.

\item The quintic forms $F$ in cusp case (of the form \eqref{CuspReducedEq1}) such that the normalization of $V(F) \subset \mathbb{P}^2$ has $30$ $\mathbb{F}_9$-rational points are the following:
\begin{eqnarray}
F_{18} & = & x^2 z^3 + (x^2 y + y^3) z^2 + (x^4 + 2 x^2 y^2)z + x^4 y + y^5; \nonumber \\
F_{19} & = & x^2 z^3 + (x^2 y + 2 y^3) z^2 + (2 x^4 + 2 x^2 y^2) + 2 x^4 y + 2 y^5; \nonumber \\
F_{20} & = & x^2 z^3 + (x^2 y + 2 y^3) z^2 + (x^4 + x^3 y  + 2 x^2 y^2)z + x^5 + 2 x^3 y^2 + 2 x^2 y^3 + x y^4 + 2 y^5; \nonumber \\ 
F_{21} & = & x^2 z^3 + (x^2 y + 2 y^3) z^2 + (x^4 + 2 x^3 y + 2 x^2 y^2)z + 2 x^5 + x^3 y^2 + 2 x^2 y^3 + 2 x y^4 + 2 y^5; \nonumber \\
F_{22} & = & x^2 z^3 + (x^3 + x^2 y + y^3) z^2 + (2 x^3 y + 2 x^2 y^2 )z + x^5 + 2 x^3 y^2 + x^2 y^3 + x y^4 + y^5. \nonumber
\end{eqnarray}
\end{itemize}

For the above $22$ curves, we classify their $\mathbb{F}_9$-isomorphism classes, by an algorithm provided in \cite{trigonal} (see Remark \ref{rem:isom} (2) below for a brief review).
We implemented and executed a procedure similar to Step 2 of Procedure \ref{proc:hyp}.
The text files \texttt{log\_trigonal\_g5q9\_step2.txt} and \texttt{log\_trigonal\_g5q9\_step2.txt} are the source code of the procedure and its log file respectively, and it took within a second to execute the code.
As a result, there are eight $\mathbb{F}_{9}$-isomorphism classes among them:\\[2mm]
\begin{tabular}{clcl}
(1) & $F_1$, $F_4$ and $F_9$ with Weil polynomial \eqref{Weil:1}; &
(2) & $F_2$, $F_6$ and $F_8$ with Weil polynomial \eqref{Weil:2}; \\[2mm]
(3) & $F_3$, $F_5$ and $F_7$ with Weil polynomial \eqref{Weil:3}; &
(4) & $F_{10}$ and $F_{17}$ with Weil polynomial \eqref{Weil:4}; \\[2mm]
(5) & $F_{11}$ and $F_{12}$ with Weil polynomial \eqref{Weil:5}; &
(6) & $F_{13}$ and $F_{14}$ with Weil polynomial \eqref{Weil:6}; \\[2mm]
(7) & $F_{15}$ and $F_{16}$ with Weil polynomial \eqref{Weil:7}; & & \\[2mm]
\end{tabular}
\begin{tabular}{clcl}
(8) & $F_{18}$, $F_{19}$, $F_{20}$, $F_{21}$ and $F_{22}$ with Weil polynomial \eqref{Weil:8}, & & \\[2mm]
\end{tabular}

\noindent where we computed the Weil polynomials of the $8$ curves by a method described in Remark \ref{rem:WeilPoly}.
(At \cite{HPkudo}, the source code for computing the Weil polynomials and its log file are named \texttt{trigonal\_g5q9\_WeilPoly.txt} and \texttt{log\_trigonal\_g5q9\_WeilPoly.txt} respectively at \cite{HPkudo}, and the time it took to execute the code is about 40 seconds.)

~\qed
% \begin{enumerate}
% \item $F_1$, $F_4$ and $F_9$ with Weil polynomial \eqref{Weil:1};
% \item $F_2$, $F_6$ and $F_8$ with Weil polynomial \eqref{Weil:2};
% \item $F_3$, $F_5$ and $F_7$ with Weil polynomial \eqref{Weil:3};
% \item $F_{10}$ with Weil polynomial \eqref{Weil:4};
% \item $F_{11}$ and $F_{12}$ with Weil polynomial \eqref{Weil:5};
% \item $F_{13}$ and $F_{14}$ with Weil polynomial \eqref{Weil:6};
% \item $F_{15}$ and $F_{16}$ with Weil polynomial \eqref{Weil:7};
% \item $F_{17}$ with Weil polynomial \eqref{Weil:4};
% \item $F_{18}$ and $F_{22}$ with Weil polynomial \eqref{Weil:8};
% \item $F_{19}$, $F_{20}$ and $F_{21}$ with Weil polynomial \eqref{Weil:8}.
% \end{enumerate}
% \end{proof}

% \begin{rem}\label{thm:main2-2}
% Our method is applicable to the case of $\F_3$-rational points.
% The maximal number of $\#C(\mathbb{F}_{3})$ of genus-five trigonal curves $C$ over $\mathbb{F}_3$ is $12$.
% Moreover, there are exactly nine $\mathbb{F}_{3}$-isomorphism classes of genus-five trigonal curves $C$ over $\mathbb{F}_3$ with $12$ $\mathbb{F}_{3}$-rational points. Due to the space limitation, %For reasons of space
% we omit the proof and explicit defining equations of the curves.
% \end{rem}

\vskip\baselineskip

We also computed the number of $\mathbb{F}_3$-rational points of each trigonal curve of genus $5$ over $\mathbb{F}_3$, and obtain the following:

\begin{prop}\label{thm:main4}
The maximal number of $\# C(\mathbb{F}_{3})$ of genus-five trigonal curves $C$ over $\mathbb{F}_3$ is $12$.
Moreover, there are exactly $9$ $\mathbb{F}_{3}$-isomorphism classes of genus-five trigonal curves $C$ over $\mathbb{F}_3$ with $12$ $\mathbb{F}_{3}$-rational points.
\end{prop}

\begin{proof}
Similarly to the proof of Theorem \ref{thm:main}, we enumerate quintic models $V(F)$ maximizing $\# C (\mathbb{F}_3)$.
For this, we conducted an exhaustive search on all possible unknown coefficients over Magma, for each form (\eqref{eq:SplitNodeQuinticGeneral}, \eqref{NonSplitNodeReducedEq1} -- \eqref{NonSplitNodeReducedEq4}, or \eqref{CuspReducedEq1}) of $F$.
The source code for collecting $F$ and its log file are available at \cite{HPkudo}, and they are named \texttt{trigonal\_g5q3\_step1.txt} and \texttt{log\_trigonal\_g5q3\_step1.txt}.
It took about 13 hours to execute the code.

We found from the output (\texttt{log\_trigonal\_g5q3\_step1.txt}) that the maximal value of $\# C (\mathbb{F}_3) $ is $12$.
In the following, we list quintic forms $F$ such that $\#C (\mathbb{F}_3)=12$:
\begin{itemize}
\item The quintic forms $F$ in split node case (of the form \eqref{eq:SplitNodeQuinticGeneral}) such that the normalization of $V(F) \subset \mathbb{P}^2$ has $12$ $\mathbb{F}_3$-rational points are the following:
\begin{eqnarray}
F_1 & = & x y z^3 + (x^3 + x^2 y + 2 x y^2 + 2 y^3) z^2 + (2 x^2 y^2 + 2 x y^3 + y^4) z + 2 x^5 + x^3 y^2; \nonumber \\
F_2 & = & x y z^3 + (x^3 + x^2 y + 2 x y^2 + 2 y^3) z^2 + (x^2 y^2 + 2 x y^3 + 2 y^4) z + 2 x^5 + 2 x^4 y + x^3 y^2 + x^2 y^3 \nonumber \\
F_3  & = & x y z^3 + (x^3 + x^2 y + 2 x y^2 + 2 y^3) z^2 + (2 x^2 y^2 + 2 x y^3 + y^4) z + 2 x^5 + x^4 y + x^3 y^2 + 2 x^2 y^3; \nonumber \\
F_4 & = & x y z^3 + (x^3 + x^2 y + 2 x y^2 + 2 y^3) z^2 + (x^2 y^2 + 2 x y^3 + 2 y^4) z + 2 x^5 + x^4 y + x^3 y^2 + 2 x^2 y^3; \nonumber \\
F_5 & = & x y z^3 + (x^3 + x^2 y + 2 x y^2 + 2 y^3) z^2 + (x^2 y^2 + 2 x y^3 + 2 y^4) z + 2 x^5 + x y^4; \nonumber \\
F_6 & = & x y z^3 + (x^3 + x^2 y + 2 x y^2 + 2 y^3) z^2 + (2 x^2 y^2 + 2 x y^3 + y^4) z + 2 x^5 + 2 x^4 y + x^2 y^3 + x y^4; \nonumber \\
F_7 & = & x y z^3 + (x^3 + x^2 y + 2 x y^2 + 2 y^3) z^2 + (x^2 y^2 + 2 x y^3 + 2 y^4) z + 2 x^5 + 2 x^4 y + x^2 y^3 + x y^4; \nonumber \\
F_8 & = & x y z^3 + (x^3 + x^2 y + 2 x y^2 + 2 y^3) z^2 + (2 x^2 y^2 + 2 x y^3 + y^4) z + 2 x^5 + x^4 y + 2 x^2 y^3 + x y^4; \nonumber \\
F_9 & = & x y z^3 + (x^3 + x^2 y + 2 x y^2 + 2 y^3)z^2 + (2 x^2 y^2 + 2 x y^3 + y^4) z + 2 x^5 + 2 x^3 y^2 + 2 x y^4; \nonumber \\
F_{10} & = & x y z^3 + (x^3 + x^2 y + 2 x y^2 + 2 y^3) z^2 + (x^2 y^2 + 2 x y^3 + 2 y^4) z + 2 x^5 + 2 x^4 y + 2 x^3 y^2 + x^2 y^3 + 2 x y^4; \nonumber \\
F_{11} & = & x y z^3 + (x^3 + x^2 y + 2 x y^2 + 2 y^3) z^2 + 2 x y^3 z + 2 x^5 + x^3 y^2 + 2 x^2 y^3 + y^5; \nonumber \\
F_{12} & = & x y z^3 + (x^3 + x^2 y + 2 x y^2 + 2 y^3) z^2 + 2 x y^3 z + 2 x^5 + 2 x^4 y + x y^4 + y^5; \nonumber \\
F_{13} & = & x y z^3 + (x^3 + x^2 y + 2 x y^2 + 2 y^3) z^2 + 2 x y^3 z + 2 x^5 + x^4 y + x^2 y^3 + x y^4 + y^5; \nonumber \\
F_{14} & = & x y z^3 + (x^3 + x^2 y + 2 x y^2 + 2 y^3) z^2 + 2 x y^3 z + 2 x^5 + 2 x^2 y^3 + x y^4 + y^5; \nonumber \\
F_{15} & = & x y z^3 + (x^3 + x^2 y + 2 x y^2 + 2 y^3) z^2 + 2 x y^3 z + 2 x^5 + x^4 y + 2 x^3 y^2 + x^2 y^3 + 2 x y^4 + y^5.\nonumber
\end{eqnarray}
% Note that all of the quintic forms $F_i$ with $1 \leq i \leq 15$ are of the same form \eqref{SplitNodeReducedEq1}.

\item In non-split node case, there are no $F$ that achieves $\# C (\mathbb{F}_3) = 12$.
%  see Remark \ref{rem:F3} below.

\item The quintic forms $F$ in cusp case (of the form \eqref{CuspReducedEq1}) such that the normalization of $V(F) \subset \mathbb{P}^2$ has $12$ $\mathbb{F}_3$-rational points are the following:
\begin{eqnarray}
F_{16} & = & x^2 z^3 + (x^2 y + 2 y^3) z^2 + 2 x^4 z + 2 x^2 y^3 + y^5; \nonumber \\
F_{17} & = & x^2 z^3 + (x^2 y + 2 y^3) z^2 + 2 x^4 z + 2 x^4 y + 2 x^3 y^2 + x y^4 + y^5; \nonumber \\
F_{18} & = & x^2 z^3 + (x^2 y + 2 y^3) z^2 + 2 x^4 z + 2 x^4 y + x^3 y^2 + 2 x y^4 + y^5. \nonumber
\end{eqnarray}
\end{itemize}

Similarly to the proof of Theorem \ref{thm:main}, we compute complete representatives of the $\mathbb{F}_3$-isomorphism classes of the above $18$ curves, by a procedure similar to Step 2 of Procedure \ref{proc:hyp}.
The source code for the procedure and its log file are available at \cite{HPkudo}, and they are named \texttt{log\_trigonal\_g5q3\_step2.txt} and \texttt{log\_trigonal\_g5q3\_step2.txt} respectively.
The time it took to execute the code is about 0.3 seconds.
As a result, there are nine $\mathbb{F}_3$-isomorphism classes:\\[2mm]
\begin{tabular}{clclclclcl}
$(1)$ & $F_1$ and $F_8$; & $(2)$ & $F_2$ and $F_{13}$; & $(3)$ & $F_3$; & $(4)$ & $F_4$ and $F_{15}$; & $(5)$ & $F_5$ and $F_{11}$;\\[2mm]
$(6)$ & $F_6$ and $F_9$; & $(7)$ & $F_7$ and $F_{14}$; & $(8)$ & $F_{10}$ and $F_{12}$; & $(9)$ & $F_{16}$, $F_{17}$ and $F_{18}$. & & \\[2mm]
\end{tabular}

% \begin{enumerate}
% \item $F_1$ and $F_8$;
% \item $F_2$ and $F_{13}$;
% \item $F_3$;
% \item $F_4$ and $F_{15}$;
% \item $F_5$ and $F_{11}$;
% \item $F_6$ and $F_9$;
% \item $F_7$ and $F_{14}$;
% \item $F_{10}$ and $F_{12}$;
% \item $F_{16}$, $F_{17}$ and $F_{18}$.
% \end{enumerate}
% \noindent Also over the algebraic closure $\overline{\mathbb{F}_3}$, there are nine isomorphism classes.
\end{proof}

% \begin{rem}\label{rem:F3}
% In non-split node case (of the form \eqref{NonSplitNodeReducedEq1} -- \eqref{NonSplitNodeReducedEq4}), the maximal value of $\# C (\mathbb{F}_3) $ for tigonal curves $C$ of genus five over $\mathbb{F}_3$ is $12$.
% The quintic forms $F$ in non-split node case such that the normalization of $V(F) \subset \mathbb{P}^2$ has $11$ $\mathbb{F}_3$-rational points are the following:
% \begin{eqnarray}
% F_{1} & = & (x^2 - \epsilon y^2)z^3 + (x (x^2 - \epsilon y^2) + x^3) z^2 + (2 x^4 + 2 y^4) z + x^5 + 2 x^4 y + 2 x^3 y^2 + x^2 y^3;  \nonumber \\
% F_2 & = & (x^2 - \epsilon y^2)z^3 + (x (x^2 - \epsilon y^2) + x^3) z^2 + (2 x^4 + 2 y^4) z + x^5 + x^4 y + 2 x^3 y^2 + 2 x^2 y^3;  \nonumber \\
% F_3 & = & (x^2 - \epsilon y^2)z^3 + (x (x^2 - \epsilon y^2) + x^3) z^2 + (2 x^4 + 2 y^4) z + x^5 + 2 x^4 y + x^2 y^3 + 2 x y^4;  \nonumber \\
% F_4 & = & (x^2 - \epsilon y^2)z^3 + (x (x^2 - \epsilon y^2) + x^3) z^2 + (2 x^4 + 2 y^4) z + x^5 + x^4 y + 2 x^2 y^3 + 2 x y^4,  \nonumber 
% \end{eqnarray}
% where $\epsilon = 2 \in \mathbb{F}_3^{\times} \smallsetminus (\mathbb{F}_3^{\times})^2$.
% Note that the quintic forms $F_i$ for $1 \leq i \leq 4$ are of the same form \eqref{NonSplitNodeReducedEq2}.
% \end{rem}

\begin{rem}\label{rem:isom}
In the proofs of Theorem \ref{thm:main} and Proposition \ref{thm:main4}, we used an algorithm given in \cite[Section 5.1]{trigonal}, which tests whether two trigonal curves of genus $5$ are isomorphic or not over any field $k$ containing the defining finite field $K$.
Here we review the algorithms briefly:
Let $C_1$ and $C_2$ be trigonal curves of genus $5$ over $K$, and let $V(F_1)$ and $V(F_2)$ be the associate quintics in $\mathbb{P}^2$.
Then we have that $C_1 \cong C_2$ over $k$ is equivalent to $V(F_1) \cong V (F_2)$ over $k$, i.e., there exist $M \in \mathrm{GL}_3 ( k )$ and $\lambda \in k^{\times}$ such that $M \cdot F_1 = \lambda F_2$, see \cite[Lemma 2.1.2]{trigonal}.
Regarding entries of $M$ and $\lambda$ as variables, we reduce the (non-)existence of such $M$ and $\lambda$ into that of a solution over $k$ of a multivariate system.
Similarly to the hyperelliptic case desribed in Remark \ref{rem:hyp}, the (non-)existence of such a solution is determined by computing a Gr\"{o}bner basis, see \cite[Section 5.1]{trigonal} for more details.
\end{rem}

\footnote[0]{
\noindent E-mail address of the first author: \texttt{kudo@mist.i.u-tokyo.ac.jp}\\
E-mail address of the second author: \texttt{harasita@ynu.ac.jp}
}

\end{document}